\documentclass[10pt]{amsart}
\usepackage[T1]{fontenc}
\usepackage[utf8]{inputenc}
\usepackage[english]{babel}
\usepackage{lmodern}
\usepackage{amsmath}
\usepackage{amssymb}
\usepackage{amsfonts}
\usepackage{amsxtra}
\usepackage{enumerate}
\usepackage{verbatim}
\usepackage{color}

\usepackage{graphicx}
\usepackage[mathscr]{eucal}
\usepackage[pdftex,colorlinks,urlcolor=blue,pdfstartview=FitH,
]{hyperref}

\newcommand{\N}{\mathbb N}

\newcommand{\R}{\mathbb R}

\newcommand{\mR}{\mathcal{R}}

\newtheorem{thm}{Theorem}[section]
\newtheorem{definition}[thm]{Definition}
\newtheorem{lemma}[thm]{Lemma}
\newtheorem{remark}[thm]{Remark}

\newtheorem{prop}[thm]{Proposition}

\DeclareMathOperator*{\supp}{supp}

\let\phi=\varphi
\let\e=\varepsilon

\usepackage{graphicx}
\begin{document}

\title[Existence of complete Lyapunov functions]{Existence of complete Lyapunov functions with prescribed orbital derivative}

\author{Peter Giesl}
\address{Department of Mathematics, University of Sussex, Falmer, Brighton, BN1 9QH, United Kingdom}
\email{p.a.giesl@sussex.ac.uk}

\author{Sigurdur Hafstein}
\address{Faculty of Physical Sciences, Dunhagi 5, 107 Reykjav\'ik, Iceland}
\email{shafstein@hi.is}

\author{Stefan Suhr}
\address{Fakult\"at f\"ur Mathematik, Ruhr-Universit\"at Bochum, Universit\"ats\-stra\ss e 150, 44780 Bochum, Germany}
\email{stefan.suhr@ruhr-uni-bochum.de}
\thanks{Suhr is partially supported by the SFB/TRR 191 ``Symplectic Structures in Geometry, Algebra and Dynamics'', funded by the Deutsche Forschungsgemeinschaft.}


\date{}

\dedicatory{}

\begin{abstract}Complete Lyapunov functions for a dynamical system, given by an autonomous ordinary differential equation, are scalar-valued functions that are strictly decreasing along orbits outside the chain-recurrent set. In this paper we show that we can prescribe the (negative) values of the derivative along orbits in any compact set, which is contained in the complement of the chain-recurrent set. 
Further, the complete Lyapunov function is as smooth as the vector field defining the dynamics.
This delivers a theoretical foundation for numerical methods to construct complete Lyapunov functions and renders them accessible for further theoretical analysis and development.
\end{abstract}

\maketitle

\section{Introduction}

Initial value problems of autonomous differential equations arise in many applications and define a dynamical system. Many tools have been developed to study the long-term behaviour of solutions and classify different behaviour depending on the initial conditions. One of the classical and fundamental tools is a Lyapunov function, which is a generalization of the energy in a dissipative system. It is a scalar-valued function, which is non-increasing along orbits of the dynamical system.
Complete Lyapunov functions, introduced by \cite{Auslander1964,Conley1978}, are scalar-valued functions, which are strictly decreasing along orbits outside the chain-recurrent set and satisfy additional properties for the values on the chain-recurrent set.

A complete Lyapunov function
 describes the qualitative behaviour of orbits by separating the phase space into two disjoint areas with fundamentally different behaviour of the flow: the  chain-recurrent set and its complement, where the flow is gradient-like. On the chain-recurrent set the flow is (almost) recurrent, it contains all equilibria, periodic and almost periodic orbits, as well as local attractors and repellers. The flow on the chain-recurrent set is sensitive to infinitesimal perturbations, while the gradient-like flow is robust to infinitesimal perturbations.
 Moreover, complete Lyapunov functions reveal stability properties of the chain transitive components of the chain recurrent set as well as the flow between them.

  If the complete Lyapunov function is differentiable,
  then the conditions can be expressed by the derivative along solutions, the orbital derivative: points with vanishing orbital derivative characterize the chain-recurrent set, while the orbital derivative is strictly negative in the area of gradient-like flow.

 The existence of  complete Lyapunov functions was first shown on  compact phase spaces
 \cite{Conley1978} and later on noncompact phase spaces  \cite{Hurley1991,hurley95,hurley1998,complete2011patrao}. The existence of $C^\infty$ complete Lyapunov functions on compact state spaces was shown in \cite{FaPa2019smoothingLya}, and in noncompact spaces in \cite{SuHa2020complya}. The latter proof used the connection of complete Lyapunov functions to  time functions in general relativity \cite{Hawking};  this relation was first noted by \cite{Fathi2012} and further explored in \cite{BeSu2018smmothConeLya}, which gave the first general existence results for $C^\infty$  Lyapunov functions
 on arbitrary  manifolds.

 The main condition on complete Lyapunov functions is that the orbital derivative is strictly negative in the gradient-flow part, i.e.~the complement of the chain-recurrent set. Hence, complete Lyapunov functions are  not unique and  a natural question is whether one can prescribe the values of the orbital derivative by a given negative function on the gradient-flow part.

 The main result of the paper is that, indeed, the orbital derivative can be prescribed by an arbitrary, sufficiently smooth function on any compact set, which is contained in the complement of the chain-recurrent set, see Theorem \ref{T1}. In the proof we first show that we can reduce the problem to the case where the orbital derivative is fixed to $-1$ and then we modify an existing $C^\infty$ complete Lyapunov function on the compact set, while preserving it away from it; this is achieved by modifying it on flow boxes.
The resulting complete Lyapunov function is as smooth as the vector field defining the system.

 This result has implications on the numerical construction of complete Lyapunov functions. There exist a number of numerical approaches to compute complete Lyapunov functions. One approach  divides the phase space into cells and computes the flow between these cells to construct a complete Lyapunov function \cite{Ban2006}.
Other approaches, however, fix the orbital derivative by a prescribed function and use collocation methods to solve the resulting partial differential equation for the complete Lyapunov function \cite{paper1,completeAGH2019} -- or optimization methods with a mixture of equality and inequality constraints \cite{min-quad}.  So far no existence result for these approaches using equations was available. The results of this paper can be used to ensure that numerical methods for constructing complete Lyapunov functions with prescribed orbital derivative are successful, and thus they deliver a theoretical foundation for these methods.

Let us give an overview of the paper: In Section \ref{def} we define complete Lyapunov functions and state the main result. In Section \ref{proof} we prove the results, before we conclude the paper in Section \ref{sec:conclusions}.

\section{Definition \& Main Result}
\label{def}
Let $U\subset \R^n$ be open and let $X\colon U\to\R^n$ be $C^l$ with $l\in\N\cup\{\infty\}$, where $\N:=\{1,2,3,4,5,\ldots\}$.
We consider the dynamical system defined by  solutions of the  ODE $\dot{x}=X(x)$.

\begin{definition}
The {\bf local flow} of $X$ is the map $\Phi\colon \Omega\to U$, $(t,p)\mapsto \Phi_t(p)$ such that
\begin{itemize}
\item[(i)] $\Omega\subset \R\times U$ is open with $\{0\}\times U\subset \Omega$.
\item[(ii)] For every $p\in U$ the orbit $t\mapsto \Phi_t(p)$ is the unique maximally extended solution to the initial value problem
\[
\begin{cases}
\frac{\partial}{\partial t}\Phi_t(p)&=X(\Phi_t(p)) \\
\Phi_0(p)&=p.
\end{cases}
\]
\end{itemize}
\end{definition}

\begin{remark}
\begin{itemize}
\item[(1)] The attribute ``local'' for the flow refers to local in time. We do not assume  that flowlines of $X$ exist on the whole of $\R$ for all initial values $p\in U$.
\item[(2)] With the regularity assumption on $X$ the existence of $\Phi$ is implied by the Theorem of Picard-Lindel\"off. Note that the local flow
enjoys the same regularity as the generator $X$, i.e. $\Phi\in C^l$.
\end{itemize}
\end{remark}

Let us now define the chain recurrent set. We denote by $\|.\|$ the Euclidian norm on $\R^n$.

\begin{definition}
Let $T>0$ and $\e \colon U\to (0,\infty)$ be continuous.
A finite collection of points $p_0,\ldots,p_m\in U$ $(m\ge 1)$ is an ${(\varepsilon,T)}${\bf-chain} if there exist $t_i\ge T$ with
\[
\|\Phi_{t_i}(p_i)-p_{i+1}\|\le \e(\Phi_{t_i}(p_i))
\]
for all $0\le i\le m-1$.
\end{definition}

\begin{definition}
A point $p\in U$ is {\bf chain recurrent for} $\mathbf{X}$ if  for all $T>0$ and all continuous $\e \colon U\to (0,\infty)$  there exists an $(\e,T)$-chain
$p_0=p,p_1,\ldots, p_m=p$.
\end{definition}

Denote by
$$\mathcal{R}_{X}$$
the set of chain recurrent points for $X$.

Recall that the {\bf chain transitive components} of the chain recurrent set are the equivalence classes with respect to the equivalence relation $\sim$, where
$p\sim q$ for two points $p,q\in \mR_X$ if there exists $T>0$ such that for all continuous $\e\colon U\to (0,\infty)$ there is an $(\e,T)$-chain
$p=p_0,p_1,\ldots,p_m=p$ containing $q$.

The following definition of Lyapunov functions is very closely related to \cite[Definition 1.4]{BeSu2018smmothConeLya}. Here we omit the smoothness
of the functions in favor of a lower regularity and consider only the case of vector fields. Recall that a point $p\in U$ is regular for a differentiable function
$f\colon U\to \R$ if $\nabla f(p)\neq 0$.

\begin{definition}
\label{coneLyaXYX}
Let $X\colon U\to\R^n$ be $C^l$ with $l\in\N\cup\{\infty\}$. The  function $\tau:U\to \R$ is a {\bf Lyapunov function for $\mathbf{X}$} if it is $C^l$ regular,
\[\dot{\tau}(p):=
\nabla\tau(p)\cdot X(p)\leq 0
\]
for each $p\in U$, and if, at each regular point $p$ of $\tau$, we have
$\dot{\tau}(p)<0$.
\end{definition}

\begin{remark}
This definition of a Lyapunov function is not the usual one, but particulary useful when studying numerical methods for the computation of Lyapunov functions; c.f.~\cite{opti} where similar Lyapunov functions are referred to as
\it{complete Lyapunov function candidates}.  Note that a classical $C^1$ Lyapunov function for one attractor is also a Lyapunov function in the sense above and a Lyapunov function as above, with the additional assumption
 that it is constant on the attractor where it attains its strict minimum, is a classical weak, i.e.~non-strict, Lyapunov function.
\end{remark}

In order to state the theorem we adopt the notion of complete Lyapunov function from \cite[II.\S6.4]{Conley1978}, see also
\cite[Definition 4.5]{SuHa2020complya}:

\begin{definition}
\label{comLyaDef}
A Lyapunov function $\tau\colon U\to \R$ for the vector field $X$ is {\bf complete} if it is strictly decreasing
along orbits outside of $\mR_X$  and such that (1) $\tau(\mR_X)$ is nowhere dense and (2) for $t\in \tau(\mR_X)$ the set
$\tau^{-1}(t)\cap \mR_X$ is a chain transitive component.
\end{definition}

\begin{remark}
The original definition in \cite[II.\S6.4]{Conley1978} of a complete Lyapunov function requires for $t\in \tau(\mR_X)$ the preimage $\tau^{-1}(t)$ to be a chain transitive component. In general we
cannot expect the critical levels of $\tau$ to be equal to chain transitive components. As an example consider $U=\R^2$ and a vector field
$X=\chi\cdot e_1$ with $\chi\ge 0$ and $\chi(x,y)=0$ iff $(x,y)=(0,0)$. It is obvious that the chain recurrent set $\mR_X$ consists only of the origin $(0,0)$
although the critical level $\{\tau=\tau((0,0))\}$ of any complete Lyapunov function $\tau\colon \R^2\to \R$ is strictly larger than $\{(0,0)\}$.
To see this note that $\tau(x,0)>\tau(0,0)$ for $x<0$ and $\tau(x,0)<\tau(0,0)$ for $x>0$, because $\tau$ is continuous and strictly decreasing along solution trajectories.  Thus
the critical level $\{\tau=\tau((0,0))\}$
divides the plane into at least two connected components. Since a single point does not divide the plane we arrive at the conclusion that the critical
level is strictly larger than $\{(0,0)\}$.
\end{remark}

\begin{remark}
Our definition of a complete Lyapunov function is stricter than that of Conley: $\tau$ is $C^1$, whereas Conley's function is merely continuous, and in our definition $p\in\mR_X$ implies $\nabla\tau(p)=0$, which is not necessarily the case in Conley's work, even for a differentiable $\tau$.  The advantage of this stricter definition is that the decrease condition can be written $\nabla \tau(p)\cdot X(p) <0$ for every $p\in U\setminus \mR_X$,
 which is much more accessible for numerical methods.  Note that a complete Lyapunov function from Definition \ref{comLyaDef} is also a complete Lyapunov function in the sense of Conley \cite{Conley1978}
and it was proved in \cite{SuHa2020complya} that such a complete Lyapunov function always exists, i.e.~our definition is not more restrictive.
\end{remark}

Now we are ready to state our main result:

\begin{thm}\label{T1}
Let $U\subset \R^n$ be open and let $X\colon U\to\R^n$ be $C^l$ with $l\in \N\cup\{\infty\}$. Then for every compact set $K\subset U\setminus \mathcal{R}_X$
and every $C^l$-function $g\colon U_K\to (-\infty,0)$
defined on a neighborhood $U_K$ of $K$ there exists a complete $C^l$-Lyapunov
function
\[
\tau_K\colon U\to \R
\]
with $\dot{\tau}_K|_K\equiv g$ and $\dot\tau_K<0$ on $U\setminus \mathcal{R}_X$.
\end{thm}

\begin{remark}\label{R1}
\begin{itemize}
\item[(a)] In the proof we will w.l.o.g.~assume that the local flow is complete.
Note that for every continuous function $f\colon U\to (0,\infty)$ the chain recurrent sets of $X$ and $f X$ coincide.

Further we can choose a
$C^\infty$-function $f\colon U\to \R$ with $f|_K\equiv 1$ such  that the local flow $\Psi$ of $f X$ is complete, i.e.
$$\Psi\colon \R\times U\to U,\quad (t,p)\mapsto \Psi_t(p)$$
is well defined. Note that $\Psi$ coincides with the local flow of $X$ on $K$. Further the set $\mathcal{R}_{X}=\mathcal{R}_{f X}$ is
$\Psi$-invariant. Thus proving Theorem \ref{T1} for $f X$ instead of $X$ yields the claim for the initial vector field as well. We will continue to use the
notation $X$ for the vector field under consideration.

\item[(b)]
Note that the regularity of $\tau_K$ is i.g.~optimal. As an example consider a vector field $X\colon \R^2\to \R^2$ of the form $X(x,y)=\chi(y)e_1$ for
some $C^l$-function $\chi\colon \R\to (0,\infty)$ which is nowhere $C^{l+1}$; e.g.~the $l$th derivative might be the Weierstrass function.
 Let $K:=[0,1]\times [0,1]\subset \R^2$ and $g\equiv -1$. By  Theorem
\ref{T1} we have a complete
$C^l$-Lyapunov function $\tau_K\colon \R^2\to \R$ with $\dot{\tau}_K|_K\equiv -1$. The flow of $X$ is given by $\Phi_t(x,y)=(x+t\chi(y),y)$ and
\begin{equation}
\label{OPTREG}
\tau_K(x+t\chi(y),y)-\tau_K(x,y)=-t
\end{equation}
as long as $\{(x+s\chi(y),y)\,|\; s\in[0,t]\}\subset K$. Assume that $\tau_K$ is $C^{l+1}$ on an open set $V\subset K$. Choose $0<t_0$, $(x_0,y_0)$ and $\delta>0$ such
that
\[
(x+t\chi(y),y)\in V
\]
for all $t\in [0,t_0]$ and all $|x-x_0| < \delta$ and $|y-y_0| < \delta$. Since $\tau_K$ is $C^{l+1}$ by assumption and for all $t$ and $(x,y)\in V$ in question, we have by \eqref{OPTREG}
for all small enough $h>0$ that
$$
\frac{\tau_K(x+h,y)-\tau_K(x,y)}{h}=\frac{-1}{\chi(y)},
$$
in particular
$$
\partial_1\tau_K(x,y)=\lim_{h\to 0+}\frac{\tau_K(x+h,y)-\tau_K(x,y)}{h}=\frac{-1}{\chi(y)} \neq 0.
$$
Further,
 the level sets $\{\tau_K=\tau_K(x_0,y_0)\}$ and $\{\tau_K=\tau_K(x_0+t_0\chi(y_0),y_0)\}$ are
graphs
\[
\{(\phi_i(y),y)|\; y\in I\},\quad i=0,1
\]
of two $C^{l+1}$ functions $\phi_0\colon I\to \R$ and $\phi_1\colon I\to\R$ respectively, where $I\subset \R$ is a sufficiently small interval around $y_0$.
 Hence
 $$
 \frac{-t_0}{\phi_1(y)-\phi_0(y)}=\frac{\tau_K(\phi_1(y),y)-\tau_K(\phi_0(y),y)}{\phi_1(y)-\phi_0(y)}=\frac{-1}{\chi(y)},
 $$
 i.e.~$\phi_1(y)-\phi_0(y)=t_0\chi(y)$ for all $y\in I$, from which $\chi$ in $C^{l+1}$ on $I$follows, a contradiction to the assumption.
\end{itemize}
\end{remark}

\section{Proof of Theorem \ref{T1}}
\label{proof}

The proof consists of modifying a sufficiently fast descending Lyapunov function on $K$ while preserving it away from $K$. This is accomplished in
Proposition \ref{P1}. The proposition in turn relies on the main technical Lemma \ref{L1} which gives the construction on a single flow box (see
definition below). The proof of the proposition is then a repeated application of the lemma.

By \cite{SuHa2020complya} we know that $X$ admits a complete $C^\infty$-Lyapunov function, $\tau'\colon U\to \R$.
We define {\bf flow boxes} as follows. For $r\in \R$ let $V_{\tau',r}$ be a precompact relatively open subset of $ \{\tau'=r\}\setminus \mathcal{R}_X$, i.e.
$V_{\tau',r}$ is open in $\{\tau'=r\}\setminus \mR_X$ and the closure $\overline{V_{\tau',r}}$ is a compact subset of $U\setminus \mathcal{R}_X$. The map
$$\Phi\colon \R\times V_{\tau',r}\to U\setminus \mR_X,\quad (t,q)\mapsto \Phi_t(q)$$
is a diffeomorphism onto its image. For $T>0$ the set
$$\mathcal{V}_{\tau',r,T}:=\Phi((-T,T)\times V_{\tau',r})\subset U\setminus \mathcal{R}_X$$
is called a {\it flow box}, cf.~Figure \ref{flowbox}.
\begin{figure}[ht]
	\centering
 \includegraphics[width=0.7\textwidth]{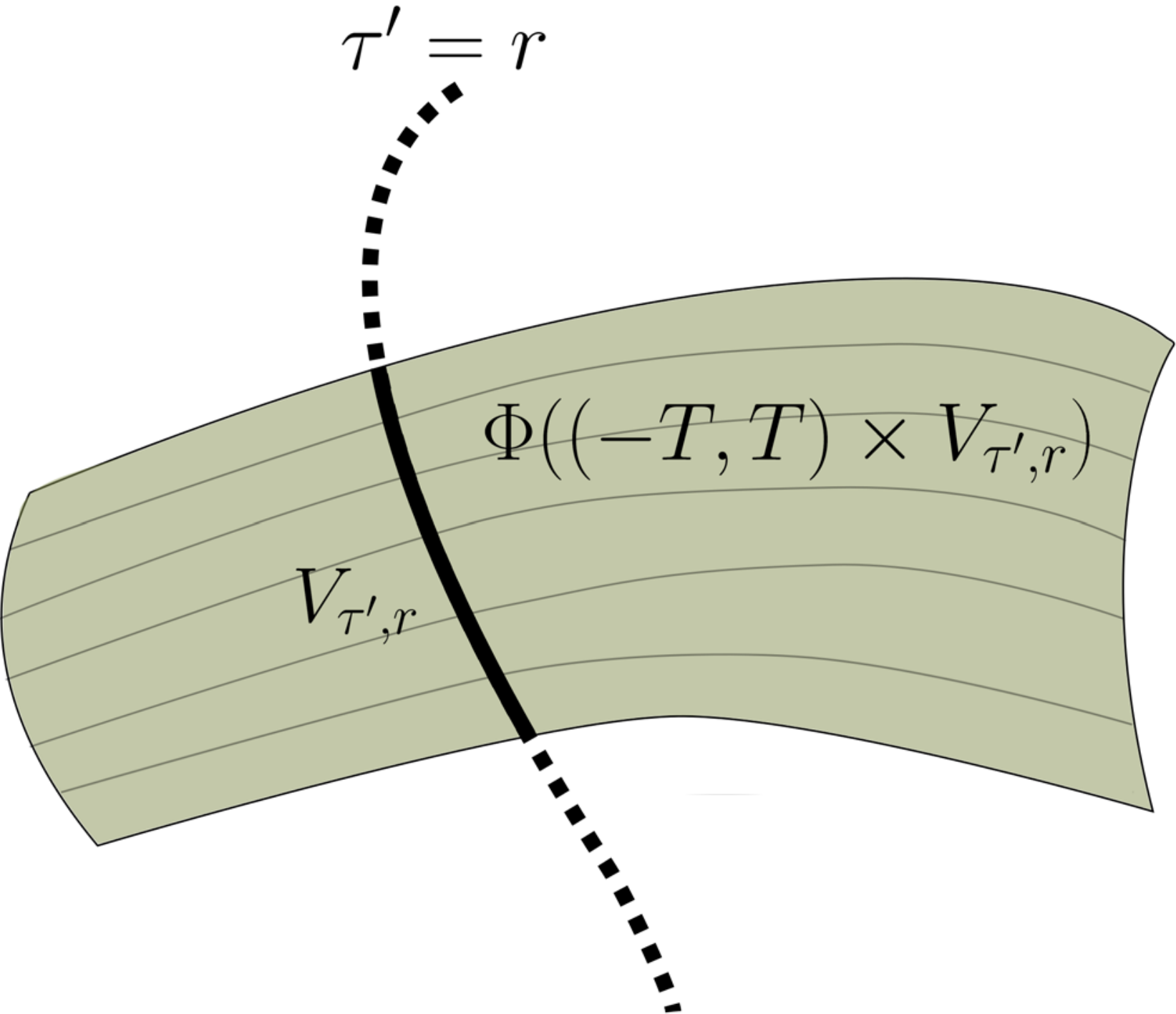}
	\caption{Schematic figure of a flow box $\mathcal{V}_{\tau',r,T}$.}
	\label{flowbox}
\end{figure}

Choose $r_1>\ldots >r_N\in \R$ and $V_{\tau',r_i}\subset \{\tau'=r_i\}\setminus \mathcal{R}_X$ such that the flow boxes
$$\mathcal{V}_{\tau',r_i,1}:=\Phi((-1,1)\times V_{\tau',r_i})$$
form an open cover of $K$, i.e.
\[
K\subset \bigcup_{i=1}^N \mathcal{V}_{\tau',r_i,1}.
\]
Choose  precompact relatively open subsets $\overline{V_{\tau',r_i}}\subset W_{\tau',r_i}\subset \{\tau'=r_i\}$.
Then the flow boxes satisfy
$$\overline{\mathcal{V}_{\tau',r_i,1}}\subset \mathcal{W}_{\tau',r_i,i+1}=\Phi((-(i+1),i+1)\times W_{\tau',r_i}).$$
Choose a constant $0<C<\infty$ such that $C\dot{\tau}'<-N-3$ on $\bigcup_{i=1}^N \mathcal{W}_{\tau',r_i,i+1}$.
Set $\tau:=C\tau'$, $s_i:= Cr_i$, $V_{s_i}:=V_{\tau',r_i}$, $W_{s_i}:=W_{\tau',r_i}$, $\mathcal{V}_{s_i,1}:=\mathcal{V}_{\tau',r_i,1}$,
and $\mathcal{W}_{s_i,i+1}:=\mathcal{W}_{\tau',r_i,i+1}$.
We then have
\begin{equation}\label{E2}
\dot{\tau}(p)<-N-3\text{ for all }p\in \bigcup_{i=1}^N \mathcal{W}_{s_i,i+1}.
\end{equation}

We will deduce Theorem \ref{T1} from the following modification result.

\begin{prop}\label{P1}
Let $X\colon U\to \R^n$ be $C^l$ with $l\in\N\cup\{\infty\}$ and let $K\subset U\setminus \mR_X$ compact be given. If \eqref{E2} holds, then there exists a complete
$C^l$-Lyapunov function $\overline{\tau}_K\colon U\to \R$ such that
\begin{itemize}
\item[(i)] $\overline{\tau}_K$ and $\tau$ coincide on $U\setminus \bigcup_i \mathcal{W}_{s_i,i+1}$,
\item[(ii)] $\nabla\overline{\tau}_K\cdot X\equiv -1$ on a neighborhood of $K$, and
\item[(iii)] the critical set of $\overline{\tau_K}$ is equal to $\mR_X$, i.e. $\{p\in U|\; \nabla \tau_K=0\}=\mR_X$.
\end{itemize}
\end{prop}

\begin{proof}[Proof of Theorem \ref{T1}]
From Proposition \ref{P1} we obtain a complete Lyapunov function $\overline{\tau}_K$ which satisfies all claims of Theorem \ref{T1} except
$\nabla\overline{\tau}_K\cdot X|_K\equiv -1$ and not $\nabla\overline{\tau}_K\cdot X|_K\equiv g$ for a given  $g\colon U_K\to (-\infty,0)$.
Choose a closed neighborhood $V_K$ of $U\setminus U_K$
disjoint from $K$. Extend $g$ to a negative $C^l$-function on $U$ with $g|_{V_K}\equiv -1$ and consider the vector field $X_g:=-X/g$.
Choose
a Lyapunov function $\tau_K$ for $X_g$ according to Proposition \ref{P1}. Then we have $\nabla\tau_K\cdot X_g|_K\equiv -1$ which is
equivalent to $\nabla \tau_K\cdot X|_K=\dot{\tau}_K|_K\equiv g$.
\end{proof}

Proposition \ref{P1} follows from the next lemma by repeated application.

\begin{lemma}\label{L1}
Let $\tau\colon U\to\R$ be a complete $C^l$-Lyapunov function and $M\subset U\setminus \mathcal{R}_X$ be closed and assume that
$\nabla\tau\cdot X\equiv -1$ on a neighborhood of $M$.
Let $V_s,W_s\subset \{\tau=s\}\setminus\mR_X$ be relatively open precompact sets with $\overline{V_s}\subset W_s$ and $k\in\N$ such that
\begin{equation}\label{E1}
\dot{\tau}(p)<-k-3
\end{equation}
for all $p\in \Phi([k,k+1]\times \overline{W_s})$.
Then there exists a complete $C^l$-Lyapunov function $\tilde{\tau}\colon U\to\R$ with
\begin{itemize}
\item[(i)] $\tilde{\tau}\equiv \tau$ on $U\setminus \mathcal{W}_{s,k+1}$ and
\item[(ii)] $\nabla\tilde{\tau}\cdot X\equiv -1$ on a neighborhood of $\overline{\mathcal{V}_{s,1}}\cup M$,
\end{itemize}
where $\mathcal{V}_{s,1}:=\Phi((-1,1)\times V_s)$ and $\mathcal{W}_{s,k+1}:=\Phi((-(k+1),k+1)\times W_s)$ are the flow boxes around
$V_s$ and $W_s$.
\end{lemma}

\begin{proof}[Proof of Proposition \ref{P1}]
First note that since we assume that $\tau$ is a complete Lyapunov function and the critical values of $\tau$ and $\overline{\tau}_K$ coincide, it
follows trivially that $\overline{\tau}_{K}$ is also a complete Lyapunov function.

The construction of $\overline{\tau}_K$ proceeds by induction over $k=1,\ldots,N$.

For $k=1$ apply Lemma \ref{L1} to $M=\emptyset$ and $\mathcal{V}_{s,1}=\mathcal{V}_{s_1,1}$ and $\mathcal{W}_{s,2}=\mathcal{W}_{s_1,2}$.
Condition \eqref{E1} is satisfied by assumption \eqref{E2}.
This yields a Lyapunov function $\tilde{\tau}_1\colon U\to \R$ with $\tilde{\tau}_1\equiv \tau$ on $U\setminus \mathcal{W}_{s_1,2}$ and
$\nabla\tilde{\tau}_1\cdot X\equiv -1$ on a neighborhood of $\overline{\mathcal{V}_{s_1,1}}$.

Now let $k\ge 2$ and assume that a Lyapunov function $\tilde{\tau}_{k-1}\colon U\to \R$ with
$$\tilde{\tau}_{k-1}\equiv \tau\text{ on }U\setminus \bigcup_{i< k} \mathcal{W}_{s_i,i+1}$$
and
$$\nabla\tilde{\tau}_{k-1}\cdot X\equiv -1\text{ on a neighborhood of }\bigcup_{i< k}\overline{\mathcal{V}_{s_i,1}}$$
has been constructed.

Set $M=\bigcup_{i< k}\overline{\mathcal{V}_{s_i,1}}$.
We show that
\[
\Phi([k,k+1]\times \overline{W_{s_k}})\cap \bigcup_{i< k} \mathcal{W}_{s_i,i+1}=\emptyset.
\]
This especially implies $ \Phi([k,k+1]\times \overline{W_{s_k}})\cap M =\emptyset$.
Assume on the contrary that for a $p\in \overline{W_{s_k}}$ and $t_p\in [k,k+1]$
 there exist for an $i<k$ a  $q\in W_{s_i}$ and $t_q\in (-(i+1),i+1)$ such  that $\Phi(t_q,q)=\Phi(t_p,p)$.  Then we have $q=\Phi(t_p-t_q,p)$ with $t_p-t_q >0$ and
it follows that
$$
s_k < s_i=\tau(q)=\tau(\Phi(t_p-t_q,p))<\tau(p)=s_k,
$$
a contradiction.

Thus we have $\tilde{\tau}_{k-1}\equiv\tau$ on $\Phi([k,k+1]\times \overline{W_{s_k}})$. With the assumption of \eqref{E2} we conclude that
Condition \eqref{E1} is satisfied.

Now Lemma \ref{L1} with $s=s_k$ yields a complete Lyapunov function $\tilde{\tau}_k\colon U\to\R$ with $\tilde{\tau}_k\equiv \tilde{\tau}_{k-1}$ on $U\setminus
\mathcal{W}_{s_k,k+1}$, i.e. $\tilde{\tau}_k\equiv \tau$ on $U\setminus \cup_{i\le k} \mathcal{W}_{s_i,i+1}$ and
\[
\nabla\tilde{\tau}_{k}\cdot X\equiv -1\text{ on a neighborhood of }\bigcup_{i\le k}\overline{\mathcal{V}_{s_i,1}}.
\]
This finishes the induction. Setting $\overline{\tau}_K=\tilde{\tau}_N$ completes the proof.
\end{proof}

\begin{proof}[Proof of Lemma \ref{L1}]

\underline{\it Prelude:} Recall that $\mR_X$ is the set of critical points of the Lyapunov function $\tau$ and $\overline{\mathcal{W}_{s,k+1}}\subset U\setminus \mR_X$. Therefore the restriction of the flow
\[
\Phi\colon [-(k+1),k+1]\times \overline{W_s}\to \overline{\mathcal{W}_{s,k+1}}
\]
is a $C^l$-diffeomorphism. Note that since $\Phi$ is the flow of $X$ we have $X(p)=D\Phi(t,q)e_1$ for $p=\Phi(t,q)$, 
where $e_1$ is the direction of the $\R$-factor of $\R\times \overline{W_s}$. We can thus equivalently consider the Lyapunov function $\tau\circ \Phi$ for the
constant vector field $e_1$ on $[-(k+1),k+1]\times \overline{W_s}$, because
$$
\dot\tau(p)=\nabla \tau(p)\cdot X(p) = \nabla\tau(p)\cdot D\Phi(t,q)e_1=\nabla (\tau\circ \Phi)(t,q)\cdot e_1.
$$
Further, we can equivalently write the orbital derivative  $\nabla (\tau\circ \Phi)\cdot e_1$ as $\partial_{e_1}(\tau\circ \Phi)$.
Thus, we have established that it suffices to construct the function $\tilde\tau$ on $[-(k+1),k+1]\times \overline{W_s}$ such that it coincides with
$$\tau\circ \Phi\colon [-(k+1),k+1]\times \overline{W_s}\to \R$$
on a neighborhood of the relative boundary

$$
\{-(k+1)\}\times W_s\cup\{k+1\}\times W_s\cup [-(k+1),k+1]\times \partial W_s,$$
where $\partial W_s:= \overline{W_s}\setminus W_s$.
We will drop the notation $\tau\circ \Phi$ for simply $\tau$ in the following and only consider the constant vector field $e_1$.
With the same argument we replace $M$ by the preimage of $M$ under $\Phi$.  Note that we do not change $\tau$ near the boundary of $\overline{\mathcal{W}_{s,k+1}}$ and therefore do not have to consider the complement of $\overline{\mathcal{W}_{s,k+1}}$.

The proof proceeds in several steps. In the first step we construct $\tilde{\tau}$ in a neighborhood of $[-1,1]\times\overline{V_s}$.
The second step then carefully interpolates $\tilde\tau$ with $\tau$ near $\{-(k+1)\}\times W_s$ in order not to destroy the property that
$\partial_{e_1}\tau\equiv -1$ on a neighborhood of $M$. The third step then takes care of the interpolation near $\{k+1\}\times W_s$.
Finally the fourth step interpolates $\tilde{\tau}$ with $\tau$ near $[-(k+1),k+1]\times \partial W_s$ again carefully in order not to destroy
the property that $\partial_{e_1}\tau\equiv -1$ on a neighborhood of $M\cup [-1,1]\times \overline{V_s}$.

\begin{figure}[ht]
	\centering
 \includegraphics[width=\textwidth]{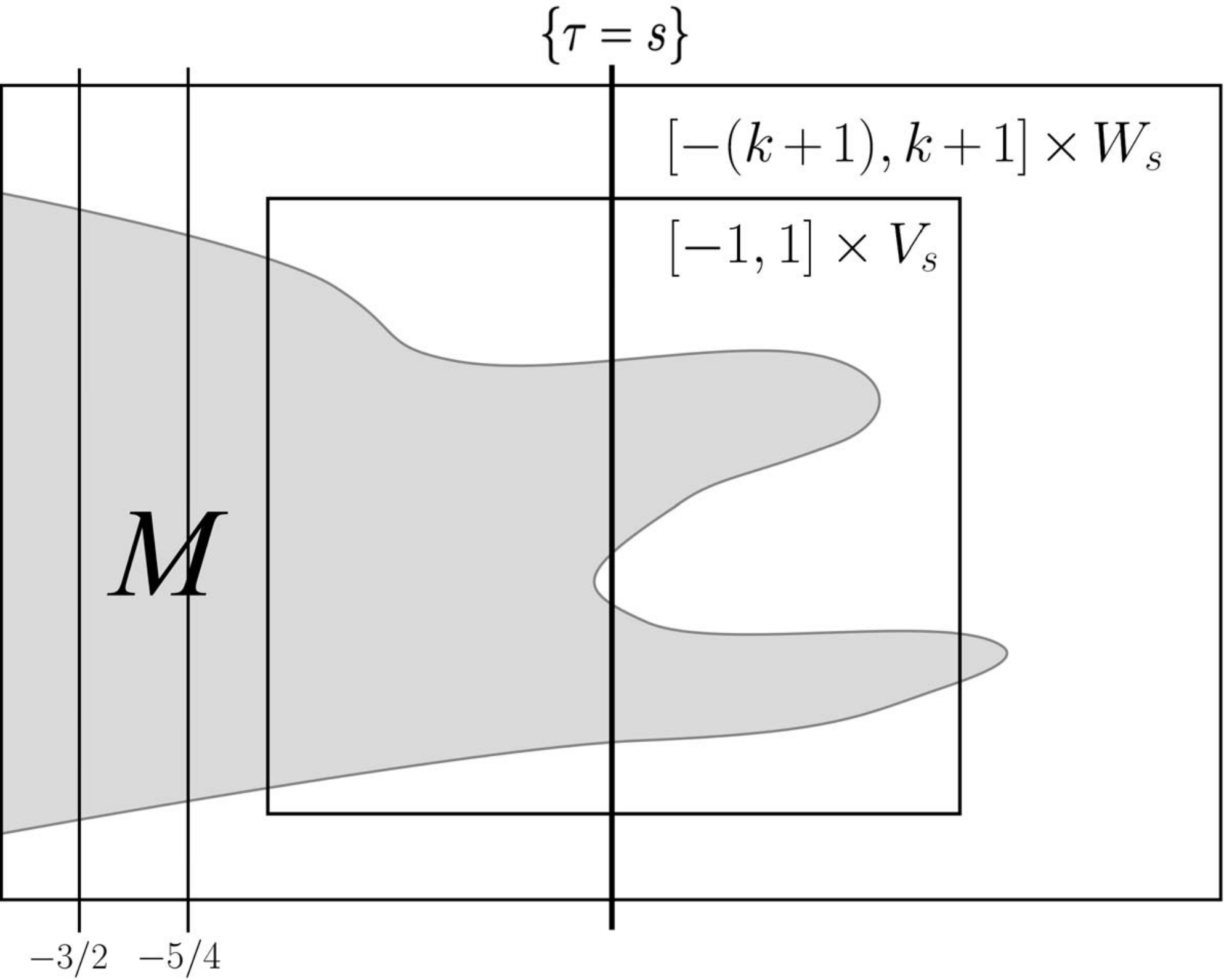}
	\caption{The first step.}
	\label{fig1}
\end{figure}

\underline{\it 1$^{st}$ step:}

In this first step we will construct a Lyapunov function $\tau_1$ on $[-(k+1),k+1]\times W_s$ with $\partial_{e_1} \tau_1 \equiv -1$ on
$[-1,1]\times \overline{V_s}$, see Figure \ref{fig1}.

Choose a smooth monotone function $\mu_-\colon \R\to[0,1]$ with:
\begin{itemize}
\item[(1)] $\mu_-\equiv 0$ for $t\le -3/2$ and
\item[(2)] $\mu_-\equiv 1$ for $t\ge -5/4$.
\end{itemize}
Define $\tau_1\colon [-(k+1),k+1]\times W_s\to \R$ by
$$\tau_1(t,q):=(1-\mu_-(t))\tau(t,q)+\mu_-(t)\left[\tau(-1,q)-\left(t+\frac{3}{2}\right)\right].$$
It is easy to see that $\tau_1$ is $C^l$-regular with $\partial_{e_1}\tau_1<0$ on $[-(k+1),k+1]\times \overline{W_s}$
and $\partial_{e_1} \tau \equiv -1$ on $[-5/4,k+1]\times \overline{W_s}$.
Indeed we have
\begin{align*}
\partial_{e_1} \tau_1=&(1-\mu_-(t))\partial_{e_1} \tau(t,q)-\mu_-(t)\\
&+ \mu'_-(t)\left[\tau(-1,q)-\left(t+\frac{3}{2}\right)-\tau(t,q)\right].
\end{align*}
The term
$$(1-\mu_-(t))\partial_{e_1} \tau(t,q)-\mu_-(t)$$
is everywhere negative, since $\tau$ is a Lyapunov function
for $e_1$ and constant to $-1$ for $t\ge -5/4$. The term
\[
\tau(-1,q)-\tau(t,q)-\left(t+\frac{3}{2}\right)
\]
is negative for $-3/2\le t \le -5/4$ because $\tau(t,q)\ge \tau(-1,q)$. Since $\mu'_-\ge 0$ and $\supp \mu'_-\subset [-3/2,-5/4]$ we conclude that  $\tau_1$ is Lyapunov. Note that $\tau_1\equiv \tau$ on $[-(k+1),-3/2]\times \overline{W_s}$.

\underline{\it 2$^{nd}$ step:} In this step, we construct a function $\tau_2$ from $\tau_1$, such that $\partial_{e_1}\tau_2 \equiv -1$ on an appropriate set involving $M$ and $[-1,1]\times \overline{V_s}$, see the later claim in this step and Figure \ref{fig2}.
\begin{figure}[ht]
	\centering
 \includegraphics[width=\textwidth]{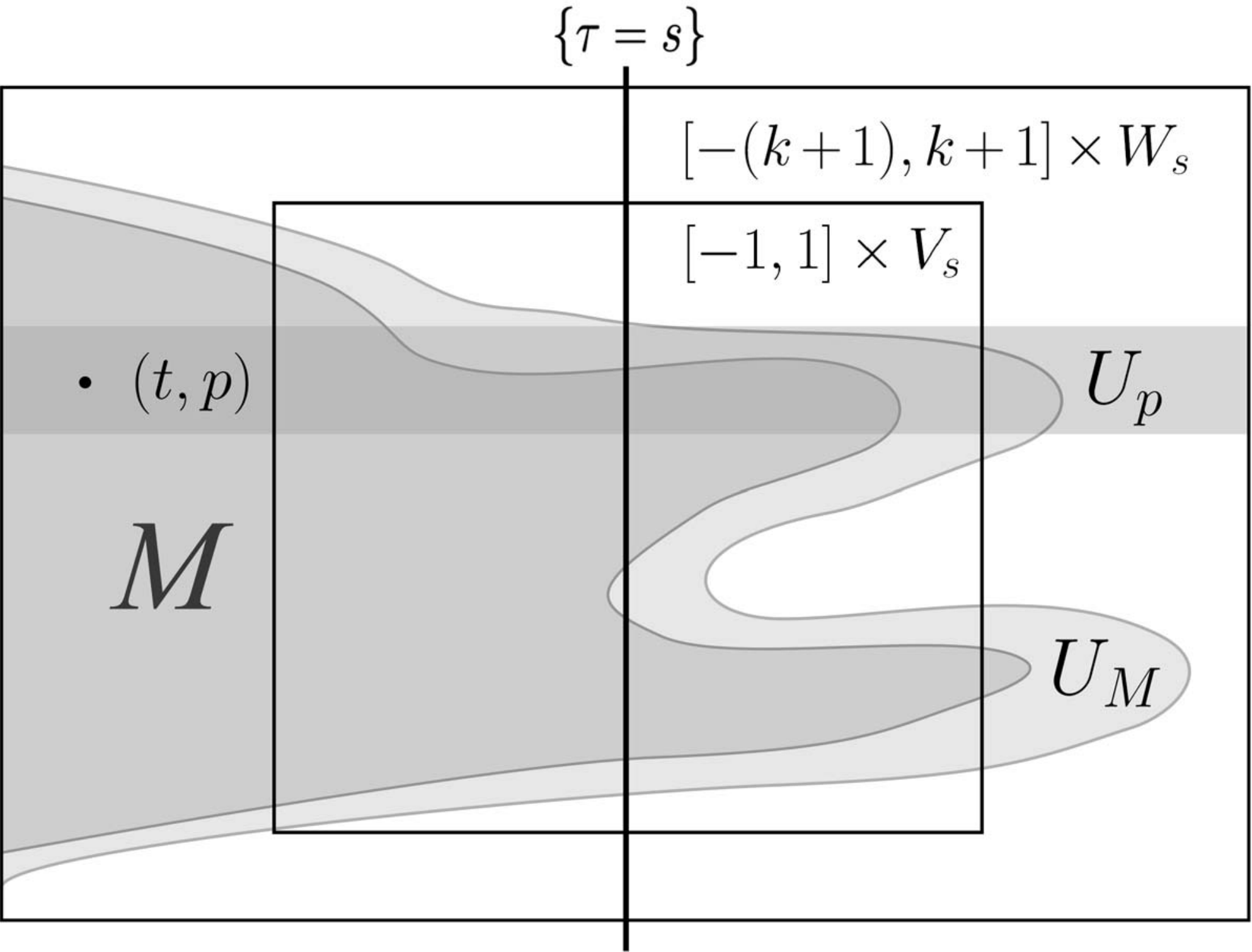}
	\caption{The second step. }
	\label{fig2}
\end{figure}

Fix a neighborhood $U_M\subset [-(k+1),k+1]\times \overline{W_s}$ of $M$ (in the relative topology), such that
$\partial_{e_1}\tau|_{U_M}\equiv -1$.
For $(t,p)\in M\cap [-3/2,-5/4]\times \overline{W_s}$ consider the level set
$$\{\tau=\tau(t,p)\}\cap [-7/4,-1]\times \overline{W_s}.$$
Choose a neighborhood $U_{p}$ of $p$ in $\overline{W_s}$ such that there exists an open interval $I_p$ containing $\tau(t,p)$ and
\[
\{\tau\in I_p\}\cap [-7/4,-1]\times \overline{U_p}\subset U_M.
\]

Note that $\partial_{e_1}\tau\equiv -1$ on $\{\tau\in I_p\}\cap [-7/4,-1]\times \overline{U_p}$ according to the choice of $U_M$.
By the Implicit Function Theorem there exists a $C^l$-function
$$\phi_p\colon U_p\to [-7/4,-1]$$
with $\tau(u,q)=\tau(t,p)$ iff $u=\phi_p(q)$, i.e. a parameterization of a part of the level set $\{\tau=\tau(t,p)\}$ (if necessary, shrink $U_p$).
Note that we can assume (possibly after further shrinking $U_p$) that
\begin{equation}\label{E4}
\tau(u,q)=\tau(\phi_p(q),q)-u+\phi_p(q)
\end{equation}
in a neighborhood of $\{(\phi_p(q),q)|\; q\in U_p\}$ since the points $(\phi_p(q),q)$ belong to $U_M$  and $\partial_{e_1}\tau\equiv -1$ on $U_M$.
Define a function
\begin{align*}
\sigma_p\colon [-(k+1),k+1]\times U_p&\to \R,\\
\sigma_p(u,q)&:=\begin{cases}
\tau(u,q),&\text{ for }u\le \phi_p(q)\\
\tau(\phi_p(q),q)-u+\phi_p(q)&\text{ for }u\ge \phi_p(q).
\end{cases}
\end{align*}
The function $\sigma_p$ is $C^l$-regular by \eqref{E4}. Further we have
$\partial_{e_1} \sigma_p<0$ everywhere with $\equiv -1$ on $\{(u,q)|\; u\ge \phi_p(q)\}\subset [-(k+1),k+1]\times U_p$.

We  select a finite subcover $\{U_j:=U_{p_j}\}_{j=1,\ldots, R}$ of the compact set
$$M_s:=\{p\in \overline{W_s}|\; \exists t\in [-7/4,-5/4]: (t,p)\in M\}.$$
Let $\{\lambda_j\}_j$ be a smooth partition of unity subordinate to $\{U_j\}_j$. Then
\begin{align*}
\sigma\colon [-(k+1),k+1]\times \bigcup_{1\le j \le R}  U_j&\to \R, \\
\sigma(t,q)&:=\sum_j \lambda_j(q) \sigma_{p_j}(t,q)
\end{align*}
is a $C^l$-function with $\partial_{e_1} \sigma \equiv -1$ on $[-1,k+1]\times(\cup_j  U_j)$. Note that $\sigma\equiv \tau$ on $[-(k+1),-7/4]\times (\cup_j U_j)$.

Let $\mathbb{U}_1$ be a neighborhood of $M_s$ (in the relative topology of $\overline{W_s}$) with closure in $\cup_j  U_j$ and let $\nu_1\colon \overline{W_s}\to [0,1]$ be smooth with
$\nu_1|_{\mathbb{U}_1}\equiv 1$ and $\supp\nu_1\subset \cup_j  U_j$.

Now the function
\begin{align*}
\tau_2\colon [-(k+1),k+1]\times \overline{W_s}&\to \R,\\
\tau_2(t,q)&:= \nu_1(q) \sigma(t,q) +[1-\nu_1(q)]\tau_1(t,q)
\end{align*}
is $C^l$ and Lyapunov for $e_1$.

{\it Claim:} We claim that $\partial_{e_1} \tau_2\equiv -1$ on
a neighborhood of
$$M\cup [-1,1]\times \overline{V_s}.$$

{\it Proof of the claim:}
We have $\partial_{e_1} \sigma \equiv -1$ on a neighborhood of $[-1,k+1]\times \supp \nu_1$ and $\partial_{e_1} \tau_1\equiv -1$ on $[-5/4,k+1]\times
\overline{W_s}$. Therefore
$$
\partial_{e_1} \tau_2 (t,q) = \nu_1(q) \partial_{e_1} \sigma(t,q)+(1-  \nu_1(q))\partial_{e_1} \tau_1(t,q)\equiv -1
$$
and the claim is obvious on a neighborhood of $[-1,k+1]\times \overline{W_s}$.  In particular we obtain $\partial_{e_1}
\tau_2\equiv -1$
 on a neighborhood of
$(M\cap [-1,k+1]\times \overline{W_s})\cup [-1,1]\times \overline{V_s} \subset [-1,k+1]\times \overline{W_s}$.

It only remains to consider the set $M\cap [-(k+1),-1)\times \overline{W_s}$ because $M \subset  [-(k+1),k+1]\times \overline{W_s}$.

For $(t,p)\in M$ with $t\in (-5/4,-1)$  there are two cases: If $p\notin \supp \nu_1$ we have $\tau_2\equiv \tau_1$ in a neighborhood of $(t,p)$. It follows
that $\partial_{e_1} \tau_2 \equiv \partial_{e_1} \tau_1\equiv -1$ in a neighborhood of $(t,p)$. If $p\in \supp \nu_1$ note that $\partial_{e_1} \sigma_{p_j}
\equiv -1$ in a neighborhood of $(t,p)$ for all $1\le j\le N$ such that $p\in U_j$. This implies $\partial_{e_1}\sigma\equiv -1$ near $(t,p)$.
Since $\partial_{e_1}\tau_1\equiv -1$ in a neighborhood
of $(t,p)$ we obtain again $\partial_{e_1} \tau_2 \equiv -1$ in a neighborhood of $(t,p)$.

For $(t,p)\in M$ with $t\in [-3/2,-5/4]$ we have $\nu_1\equiv 1$ near $p$. As in the previous case we have $\partial_{e_1}\sigma\equiv -1$
in a neighborhood of $(t,p)$, i.e. $\partial_{e_1}\tau_2\equiv -1$ near $(t,p)$.

Finally for $(t,p)\in M$ with $t\in [-(k+1),-3/2)$ we again distinguish two cases: First assume $p\notin \supp \nu_1$. Then $\tau_2\equiv \tau_1\equiv
\tau$ in a neighborhood of $(t,p)$. Since $(t,p)\in M$ we conclude $\partial_{e_1}\tau_2\equiv -1$ near $(t,p)$.  Now assume
$p\in\supp \nu_1$. For $1\le i\le N$ such that $\phi_i(p)$ is defined, i.e.~$p\in U_{p_i}$, and $t<\phi_i(p)$ we have $\sigma_{p_i}\equiv \tau$ near $(t,p)\in M$, i.e. $\partial_{e_1}
\sigma_{p_i}\equiv -1$ near $(t,p)$.
For $1\le i\le N$ such that $\phi_i(p)$ is defined and $t\ge \phi_i(p)$ we have $\partial_{e_1}\sigma_{p_i}\equiv -1$ in a neighborhood of $(t,p)$ trivially by construction.
Since $\tau_1\equiv \tau$ near $(t,p)$ and $(t,p)\in M$ we also have $\partial_{e_1}\tau_1\equiv -1$ near $(t,p)$.
Summing up we conclude $\partial_{e_1}\tau_2\equiv -1$ in a neighborhood of $(t,p)$.

This concludes the proof of the claim.

\underline{\it 3$^{rd}$ step:}
Next we modify $\tau_2$ on $[k,k+1]\times \overline{W_s}$ so that it coincides with $\tau$ near $\{k+1\}\times \overline{W_s}$, see Figure \ref{fig3}.
\begin{figure}[ht]
	\centering
 \includegraphics[width=\textwidth]{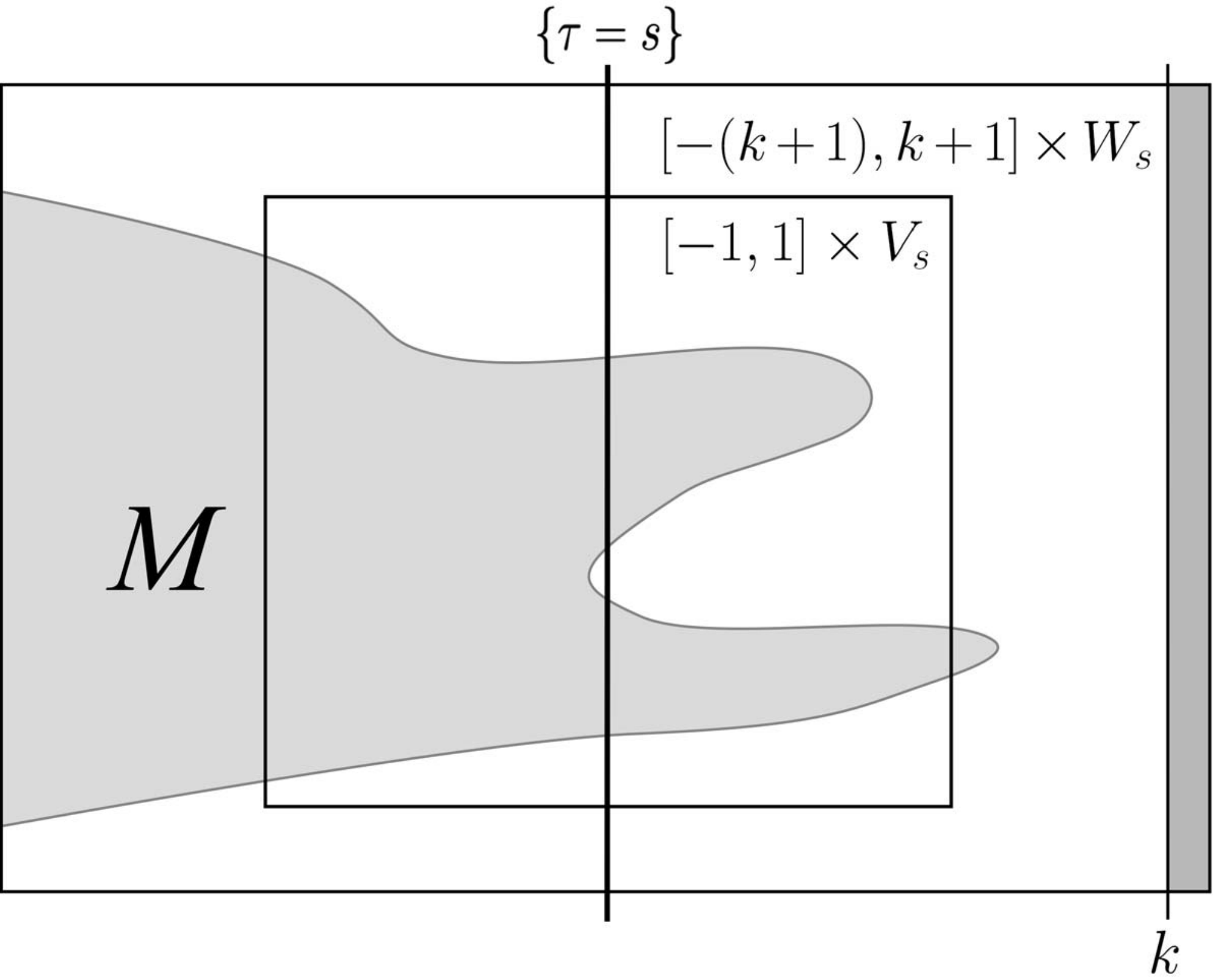}
	\caption{The third step }
	\label{fig3}
\end{figure}
We start with estimating $\tau_2(k+1,q)$ from below. Note that by construction
\[
\tau_1(k+1,q)=\tau(-1,q)-k-5/2
\]
and for   $q\in U_j$ we have
\begin{align*}
\sigma_{p_j}(k+1,q)&=\tau(\phi_{p_j}(q),q)-(k+1)+\phi_{p_j}(q)\\
&\ge \tau(-1,q)-(k+1)-7/4,
\end{align*}
because $\phi_{p_j}(q)\in [-7/4,-1]$.
Combining both, the definition of $\tau_2$ implies
$$\tau_2(k+1,q)\ge \tau(-1,q)-k-11/4$$
for all $q\in \overline{W_s}$. 
By \eqref{E1} we have
$$\tau(k+1,q)\le \tau(k,q)-k-3\le \tau(-1,q)-k-3$$
and therefore there exists $\e \in (0,1/2)$ such that $\tau<\tau_2$ on $[k+1-2\e,k+1]\times \overline{W_s}$. Choose a smooth monotone function $\mu_+\colon
\R\to[0,1]$ with:
\begin{itemize}
\item[(1)] $\mu_+\equiv 0$ for $t\le k+1-2\e$ and
\item[(2)] $\mu_+\equiv 1$ for $t\ge k+1-\e$
\end{itemize}
Define $\tau_3\colon [-(k+1),k+1]\times \overline{W_s}\to \R$ by
$$\tau_3(s,q):=(1-\mu_+(s))\tau_2(s,q)+\mu_+(s)\tau(s,q).$$
As before we see that $\tau_3$ is Lyapunov for $e_1$, using $\tau<\tau_2$ on the support of the derivative of $\mu_+$. Note that by assumption \eqref{E1} the sets
$M$ and $[k,k+1]\times \overline{W_s}$ are disjoint. Since $\tau_3\equiv \tau_2$ on $[-(k+1),k+1-2\epsilon]\times \overline{W_s}$ and $k<k+1-2\epsilon$
we continue to have $\partial_{e_1} \tau_3\equiv -1$ on a neighborhood of
\[
M\cup [-1,1]\times \overline{V_s}.
\]
Moreover, $\tau\equiv \tau_3$ near $\{-(k+1),k+1\}\times \overline{W_s}$.

\underline{\it 4$^{th}$ step:}
It remains to interpolate $\tau_3$ with $\tau$ near $[-(k+1),k+1]\times \partial W_s$, see Figure \ref{fig4}.
\begin{figure}[ht]
	\centering
  \includegraphics[width=\textwidth]{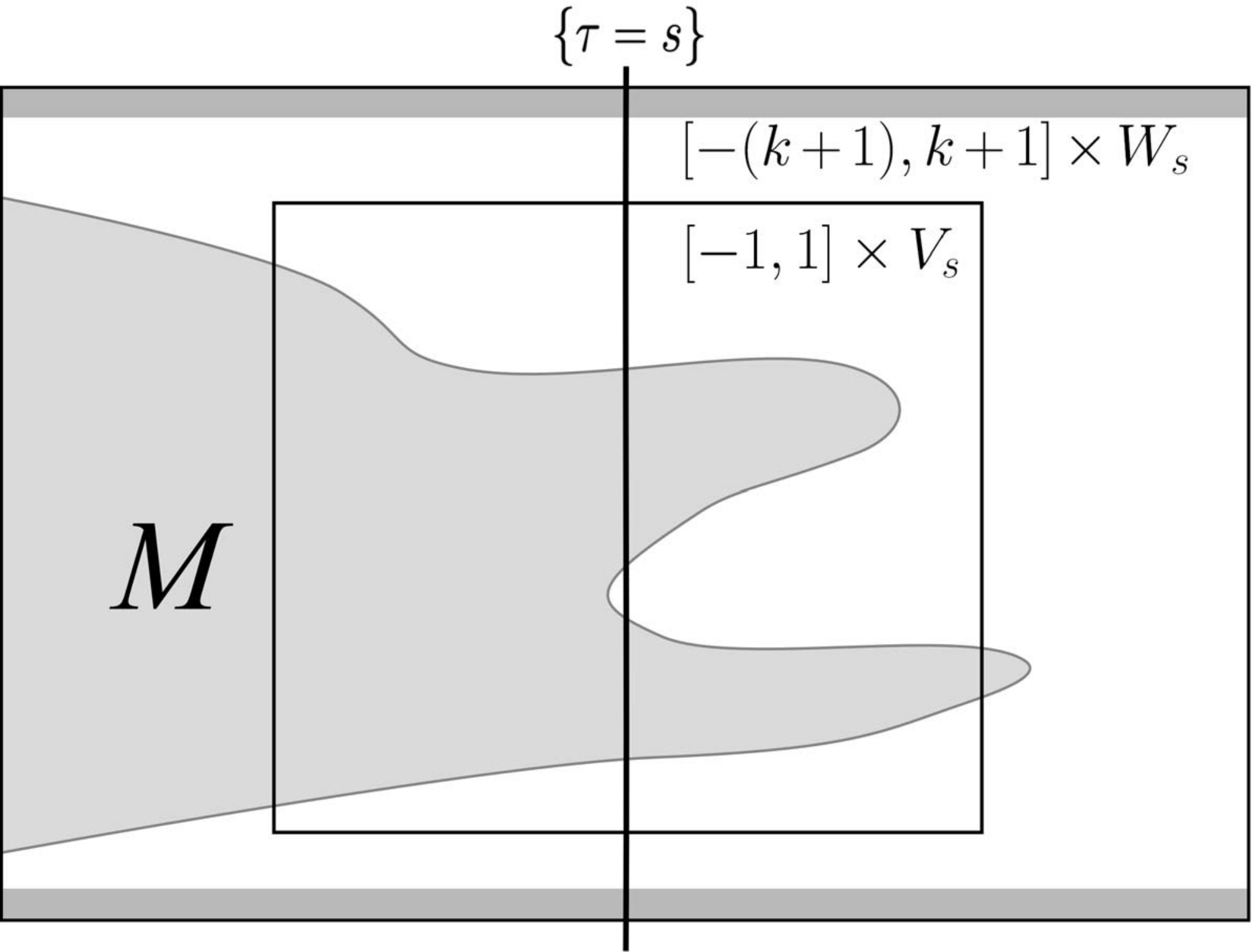}
	\caption{The fourth step }
	\label{fig4}
\end{figure}
Choose a neighborhood $\mathbb{U}_2$ of $\overline{V_s}$ with closure in $W_s$ and a smooth function
$\nu_2\colon \overline{W_s}\to [0,1]$ with $\nu_2\equiv 1$ on $\mathbb{U}_2$ and $\supp\nu_2\subset W_s$.
Then
\begin{align*}
\tau_4\colon [-(k+1),k+1]\times \overline{W_s}&\to \R,\\
\tau_4(s,q)&:=\nu_2(q)\tau_3(s,q)+(1-\nu_2(q))\tau(s,q)
\end{align*}
is a $C^l$-function which coincides with $\tau$ near the boundary of $[-(k+1),k+1]\times \overline{W_s}$ and
$\partial_{e_1}\tau_4\equiv -1$ on a neighborhood of $[-1,1]\times \overline{V_{s}}\cup M$.
Indeed the property holds for  $\tau_3$, and for $\tau$ outside of $[-1,1]\times\overline{V_s}$. Since $\nu_2|_{[-1,1]\times \overline{V_s}}\equiv 1$
the claim follows immediately. Setting $\tilde{\tau}:=\tau_4$ concludes the proof.
\end{proof}

\section{Conclusions}
\label{sec:conclusions}
We consider a dynamical system, given by the flow of a $C^l$-vector field $X\colon U \to\R^n$. For any compact subset $K$ of the complement of the chain recurrent set $U\setminus \mR_X$ and any $C^l$-function $g\colon U\to (-\infty,0)$, we have established the existence of a $C^l$-regular complete
Lyapunov function $\tau$ for the system that fulfills $\dot{\tau}(p)=\nabla \tau(p)\cdot X(p) = g(p)$ for every $p\in K$.   These results are of essential importance for methods to numerically compute complete Lyapunov functions.  Indeed, they present a major leap forward in analyzing and improving
several methods that rely on solving PDEs or convex optimization problems containing equality constraints.


\bibliographystyle{amsplain}
\bibliography{LyapBibUniform}

\end{document}